\documentclass[11pt,fleqn]{scrartcl}

\usepackage[english]{babel}
\usepackage{amsmath,amsfonts, amsthm,xcolor,amssymb}
\usepackage{paralist}
\numberwithin{equation}{section}
\usepackage{mathtools, mathabx}
\usepackage[colorlinks=true,urlcolor=blue, citecolor=blue,linkcolor=blue,linktocpage,pdfpagelabels, bookmarksnumbered,bookmarksopen]{hyperref}

\usepackage[hyperpageref]{backref}

\usepackage{xcolor}

\newtheorem{thm}{Theorem}[section]

\newtheorem{cor}[thm]{Corollary}
\newtheorem{prop}[thm]{Proposition}

\theoremstyle{definition}
\newtheorem{rem}[thm]{Remark}
\theoremstyle{remark}

\newcommand{\ds}{\displaystyle}

\newcommand{\R}{\mathbb{R}}
\newcommand{\N}{\mathbb{N}}

\newcommand{\cH}{\mathcal H}

\newcommand{\de}{\partial}
\newcommand{\eps}{\varepsilon}

\DeclareMathOperator{\dive}{div}
\DeclareMathOperator{\tr}{Tr}

\DeclareMathOperator{\Qp}{\mathcal Q_p}

\DeclareMathOperator{\arccosh}{arccosh}

\newcommand\restr[2]{{
  \left.\kern-\nulldelimiterspace 
  #1 
  \vphantom{ |} 
  \right|_{#2} 
  }}
{\left\{\begin{array}{@{}l@{}}}{\end{array}\right.}
\makeatletter 
\def\@makefnmark{} 
\makeatother 

\begin{document}

\title{Sharp estimates for the first Robin eigenvalue of nonlinear elliptic operators 
}
\author{Francesco Della Pietra$^{*}$, Gianpaolo Piscitelli$^{*}$ 
\thanks{Dipartimento di Matematica e Applicazioni ``R. Caccioppoli'', Universit\`a degli studi di Napoli Federico II, Via Cintia, Monte S. Angelo - 80126 Napoli, Italia.  \newline 
Email: f.dellapietra@unina.it (\textit{corresponding author}), gianpaolo.piscitelli@unina.it}
}
\maketitle
\begin{abstract}
\noindent{\textsc{Abstract.}} The aim of this paper is to obtain optimal estimates for the first Robin eigenvalue of the anisotropic $p$-Laplace operator, namely:
\begin{equation*}
\lambda_1(\beta,\Omega)=
\min_{\psi\in W^{1,p}(\Omega)\setminus\{0\} } \frac{\displaystyle\int_\Omega F(\nabla \psi)^p dx +\beta \ds\int_{\de\Omega}|\psi|^pF(\nu_{\Omega}) d\cH^{N-1} }{\displaystyle\int_\Omega|\psi|^p dx},
\end{equation*}
where $p\in]1,+\infty[$, $\Omega$ is a bounded, mean convex domain in $\R^{N}$, $\nu_{\Omega}$ is its Euclidean outward normal, $\beta$ is a real number, and $F$ is a sufficiently smooth norm on $\R^{N}$. The estimates we found are in terms of the first eigenvalue of a one-dimensional nonlinear problem, which depends on $\beta$ and on geometrical quantities associated to $\Omega$. More precisely, we prove a lower bound of $\lambda_{1}$ in the case $\beta>0$, and a upper bound in the case $\beta<0$. As a consequence, we prove, for $\beta>0$, a lower bound for $\lambda_{1}(\beta,\Omega)$ in terms of the anisotropic inradius of $\Omega$ and, for $\beta<0$, an upper bound of $\lambda_{1}(\beta,\Omega)$ in terms of $\beta$.
\\[.5cm]
\end{abstract}

\begin{center}
\begin{minipage}{13cm}
\small
\tableofcontents
\end{minipage}
\end{center}

\section{Introduction}
The aim of this paper is to obtain optimal estimates, in terms of a one-dimensional eigenvalue problem, for the first Robin eigenvalue of the anisotropic $p$-Laplacian operator, namely:
\begin{equation}
\label{eigint}
\lambda_1(\beta,\Omega)=
\min_{\psi\in W^{1,p}(\Omega)\setminus\{0\} } \frac{\displaystyle\int_\Omega F(\nabla \psi)^p dx +\beta \ds\int_{\de\Omega}|\psi|^pF(\nu_{\Omega}) d\cH^{N-1} }{\displaystyle\int_\Omega|\psi|^p dx},
\end{equation}
where $p\in]1,+\infty[$, $\Omega$ is a bounded, $C^{2}$ domain in $\R^{N}$, $\nu_{\Omega}$ is its Euclidean outward normal, $\beta$ is a real number, $F : \mathbb{R}^N \to [0, +\infty[$, $N\ge 2$, is a convex, even,  $1$-homogeneous and $C^{3,\alpha}(\mathbb{R}^N\setminus \{0\})$ function such that $[F^{p}]_{\xi\xi}$ is positive definite in $\R^{N}\setminus\{0\}$ (we refer the reader to Section \ref{notations_sec} for all the definitions, here and below used, related the Finsler metric $F$). When $\beta=0$, $\lambda_1(\beta,\Omega)=0$ is the first (trivial) Neumann eigenvalue, while, when $\beta$ goes to $+\infty$, it reduces to the first Dirichlet eigenvalue.

If $u\in W^{1,p}(\Omega)$ is a minimizer of \eqref{eigint}, then $u$ satisfies
\begin{equation}
    \label{pb_Robin1}
\begin{cases}
-\mathcal Q_p u =\lambda_{1} (\beta,\Omega)\left|u\right|^{p-2}u \quad &\text{in}\ \Omega,\\
F(\nabla u)^{p-1}F_{\xi}(\nabla u)\cdot \nu_{\Omega}+\beta F(\nu_{\Omega}) |u|^{p-2}u= 0 & \text{on}\ \partial\Omega,
\end{cases}
\end{equation}
where 
\[
\mathcal Q_p u=\dive(F^{p-1}(\nabla u)F_\xi(\nabla u)).
\]
We remark that when $F(\xi)=\sqrt{\sum_{i=1}^{N}\xi_{i}^{2}}$, the Euclidean norm, problem \eqref{pb_Robin1} corresponds to the classical eigenvalue problem for the Euclidean $p$-Laplace operator.

The problem of estimating \eqref{eigint} in terms of geometrical quantities related to $\Omega$ is well studied; actually, there is a striking difference between the cases $\beta \ge 0$ and $\beta <0$.

For positive values of the parameter $\beta$, if $\Omega$ is a bounded Lipschitz open set of $\R^{N}$, a Faber-Krahn inequality holds:
\begin{equation*}\label{FKintro}
\lambda_{1}(\beta,\Omega) \ge \lambda_{1}(\beta,\mathcal W)
\end{equation*}
where $\mathcal W$ is the so-called Wulff shape, with the same Lebesgue measure of $\Omega$. We refer the reader to \cite{pota} and the references therein contained.
Furthermore, in \cite{pota} it has been also proved that, in the class of convex sets,
an estimate in terms of the anisotropic inradius of $\Omega$ (that is the radius of the greatest Wulff shape that can be contained in $\Omega$), denoted by $R_F(\Omega)$, can be proved. Indeed, if $\Omega$ is a convex, bounded domain of $\R^{N}$, it holds that
\begin{equation*}
\lambda_{1}(\beta,\Omega) \ge \left(\frac{p-1}{p}\right)^{p} \frac{\beta}{R_F(\Omega)\left(1+\beta^{\frac{1}{p-1}} R_{F}(\Omega)\right)^{p-1}}
\end{equation*}
(see also \cite{K} for $p=2$ in the Euclidean case). Actually, this estimate is far to be optimal. On the other hand, in the Dirichlet case (i.e. with $\beta=+\infty$) a sharp estimate holds for smooth bounded anisotropic mean convex domains, that is with nonnegative anisotropic mean curvature:
\begin{equation}
\label{herschanona}
\lambda_1^{D}(\Omega) \ge 
(p-1)
\left(\frac{\pi_{p}}{2}\right)^{p} \frac{1}{R_{F}(\Omega)^{p}}
\end{equation}
where $\lambda_{1}^{D}(\Omega)=\lim_{\beta\to+\infty}\lambda_{1}(\beta,\Omega)$ is the first Dirichlet eigenvalue of $\mathcal Q_{p}$, and
\[
\pi_{p}=2\int_{0}^{1}(1-t^{p})^{-\frac{1}{p}}dt=\frac{2\pi}{p\sin\frac\pi p}.
\]
Moreover, the estimate \eqref{herschanona} is optimal on a suitable infinite slab. Inequality \eqref{herschanona} is a generalization, for a  large class of nonlinear operators, of a well-known inequality for the first Dirichlet eigenvalue ($p=2$) of planar convex domains in the Euclidean Laplacian case, proved by Hersch  \cite{H1}, and then generalized by Protter \cite{Pr} in the case $N\ge 2$. The Finsler-Laplacian case has been considered in \cite{bgm} for $p=2$ and convex domains, while the general case with $1<p<+\infty$ and $\Omega$ anisotropic mean convex in \cite{DPdBG}, as well as the optimality of the inequality.

Various other kind of estimates, in terms of geometrical quantities depending on $\Omega$, can be given for the eigenvalues of nonlinear operators with Robin boundary conditions, Neumann conditions, as well as in ``holed'' domains with different conditions on the various components of the boundary. We refer the reader, for example, to \cite{DPGP,DPP,PPT} and the references therein contained.

For negative values of the parameter $\beta$, the situation is very different. For $p=2$ and in the Euclidean setting, in \cite{B}, Bareket conjectured that, among all smooth bounded domains of given volume, the first eigenvalue is maximized when $\Omega$ is a ball. The conjecture remained open for a long time. In \cite{FNT} it has been showed that it is true for domains close, in a suitable sense, to a ball. Moreover, the conjecture has been disproved in \cite{FK} for $p=2$, and in \cite{KP} for $p\in]1,+\infty[$, for $\left|\beta\right|$ large enough. Furthermore, in \cite{FK} it was proved that, in two dimensions and $p=2$, the conjecture is true for small values of $\left|\beta\right|$, and in \cite{KP}, in any dimension and $1<p<+\infty$ for starshaped domains, always for $\left|\beta\right|$ small. Actually, the optimal value of $\beta$ is still unknown. As regards the Finsler setting, in \cite{PT} it has been showed that for $\left|\beta\right|$ small the Wulff shape, among all the sufficiently smooth domains with given volume, is a maximizer. 

The aim of this paper is to obtain several estimates for $\lambda_{1}(\beta,\Omega)$ in terms of geometric quantities related to $\Omega$, both in the cases $\beta\ge 0$ and $\beta <0$.

For $\beta\ge 0$, we prove an optimal lower bound for $\lambda_{1}(\beta,\Omega)$, with $\Omega$ smooth, bounded domain, which is anisotropic mean convex (that is, with nonnegative anisotropic mean curvature $\mathcal H_F$, see Section 2.2 for the precise definition), in terms of the first eigenvalue of a nonlinear, one dimensional problem.
As a consequence, we obtain an estimate in terms of the anisotropic inradius of $\Omega$, which reduces to \eqref{herschanona} when $\beta$ tends to $+\infty$. The approach we adopt makes use of the $\mathcal P$-function method, well known in the classical Euclidean case for the Laplace operator (see \cite{Sp81} and the references therein) and recently considered in the Finsler case, for example, in \cite{CFV,DPGGLB,DPdBG}. 

More precisely, we prove the following result.
\begin{thm}
\label{main1}
Let $\Omega$ be a bounded $C^{2}$ domain in $ \R^{N}$, $ N \ge 2$, such that $\de\Omega$ has nonnegative anisotropic mean curvature $\mathcal H_{F}$, and let be $\beta>0$. Then
\begin{equation}
\label{sperbine}
\lambda_{1}(\beta,\Omega) \ge \mu_1(\beta,s_{0})
\end{equation}
where $\mu_{1}(\beta,s_{0})$ is the first positive root of
\begin{equation*}
\frac{\mu}{p-1}=\frac{\beta^{p'}}{\cos_p^{-p}\left(\left(\frac{\mu}{p-1}\right)^\frac 1p s_{0} \right)-1}.
\end{equation*}
with $p'=\frac{p}{p-1}$, $s_{0}=(p'M(\Omega))^{\frac{1}{p'}}$, and $M(\Omega)$ is the maximum value of the stress function $w_\Omega$, that is the solution of 
 \begin{equation}
    \label{pbtord}
\begin{cases}
-\mathcal Q_p w_\Omega=1 \quad &\text{in}\ \Omega,\\
w_\Omega= 0 & \text{on}\ \partial\Omega.
\end{cases}
\end{equation} 
Moreover, the estimate \eqref{sperbine} is optimal for a suitable infinite slab.
\end{thm}

The definition of the generalized trigonometric function $\cos_{p}$ will be recalled in Section \ref{1D_subsec}.

Inequality \eqref{sperbine} was first proved, when $F$ is the Euclidean norm and $p=2$, by Sperb in \cite{Sp92}.
The value $\mu_{1}(\beta,s_{0})$ is the first eigenvalue of a suitable nonlinear one dimensional eigenvalue problem, namely
\begin{equation}
\label{1dimintro}
    \begin{cases}
    (|X'|^{p-2}X')'+\mu_{1}(\beta,s_{0}) |X|^{p-2}X=0 \quad\text{in}\ (0,s_0),\\
    |X'(0)|^{p-2}X'(0)=0, \\
    |X'(s_{0})|^{p-2}X'(s_{0})+\beta |X(s_{0})|^{p-2}X(s_{0})=0, 
    \end{cases}
\end{equation}
where $s_{0}= (p'M(\Omega))^{\frac{1}{p'}}$ (this problem is studied in Section 3). The inequality \eqref{sperbine} was not known also in the Euclidean case, when $p\ne 2$; we mention that optimal estimates of $\lambda_{1}(\beta,\Omega)$ with respect to the first eigenvalue of a one dimensional problem have been obtained by \cite{S} for the case of the Euclidean Laplacian eigenvalue, and \cite{LW} for the $p$-Laplacian, both on compact manifolds. In these cases, when $\Omega$ is a domain of $\R^N$, our estimate \eqref{sperbine} improves the quoted results in \cite{LW,S}. Moreover, the class of the operators we consider is very general and can be highly nonlinear.

Furthermore, Theorem \ref{main1} allows to achieve an estimate in terms of the inradius $R_F(\Omega)$, by using the following bounds for the maximum $M(\Omega)$ of the Saint-Venant torsional rigidity problem and the anisotropic inradius $R_{F}(\Omega)$.
\begin{thm}
\label{paynemax}
Let $\Omega$ be a bounded $C^{2}$ domain in $ \R^{N}$, $ N \ge 2$, such that $\de\Omega$ has nonnegative anisotropic mean curvature $\mathcal H_{F}$. Then
\[
\frac{1}{N^{p'-1}}\frac{R_{F}^{p'}(\Omega)}{p'}\leq M(\Omega)\leq\frac{R_F^{p'}(\Omega)}{p'}.
\]
 Moreover, the right-hand side inequality is optimal, as $\Omega$ approaches a suitable infinite $N$-dimensional slab; the left-hand side inequality holds as an equality if and only if $\Omega$ is a Wulff shape.
\end{thm}
 The upper bound, in the case of planar Euclidean Laplacian, goes back to a result proved in \cite{Pa68}; the two inequalities in the general $p$-Finsler case were proved in \cite{DPGGLB}.

Then as a corollary of Theorem \ref{main1}, we obtain the following explicit estimate of $\lambda_{1}(\beta,\Omega)$ in terms of $R_{F}(\Omega)$.
\begin{cor}\label{cor_H}
Under the same hypotheses of Theorem \ref{main1}, it holds that 
\begin{equation}
\label{corest}
\lambda_1(\beta,\Omega)\ge 
(p-1)\left( \frac{\pi_p}{2}\right)^p\frac{1}{\left(R_F(\Omega)+\frac{\pi_p}{2}\beta^{-\frac1{p-1}}\right)^p}
,
\end{equation}
where $R_{F}(\Omega)$ is the anisotropic inradius of $\Omega$.
\end{cor}
Incidentally, the inequality \eqref{corest} reduces to \eqref{herschanona} as $\beta \to+\infty$, since $\lambda_{1}(\beta,\Omega)$ converges to the first Dirichlet eigenvalue.

An estimate of the type \eqref{corest} has been given also in \cite{S} in the Euclidean, $p=2$ case, when $\Omega$ is a compact Riemannian manifold with smooth boundary and nonnegative Ricci curvature of $\Omega$ and mean curvature of $\de\Omega$.

In the case $\beta<0$, we get the following upper bound.

\begin{thm}\label{main2}
Let $\Omega$ be a bounded $C^{2}$ domain in $ \R^{N}$, $ N \ge 2$, with nonnegative anisotropic mean curvature $\mathcal H_{F}$ of $\de\Omega$, and let $\beta<0$. 
Then
\begin{equation}\label{sperbineN}
\lambda_{1}(\beta,\Omega) \le \mu_1(\beta,R_{F}(\Omega))
\end{equation}
where $\mu_{1}(\beta,R_{F}(\Omega))$ is the unique negative value that solves
 \begin{equation*}
-\frac{\mu}{{p-1}} =\frac{|\beta|^{p'}}{1-\cosh_p^{-p}\left(\left(\frac{-\mu}{p-1}\right)^\frac 1p R_{F}(\Omega) \right)}.
\end{equation*}
Moreover, the estimate \eqref{sperbineN} is optimal for a suitable slab.
\end{thm}

The value $\mu_{1}(\beta,R_{F}(\Omega))$ is the first eigenvalue of \eqref{1dimintro} when $\beta<0$ and $s_{0}=R_{F}(\Omega)$.

Finally, by using Theorem \ref{main2} we get the following estimate for $\lambda_{1}$ in terms of $\beta$.

\begin{cor}\label{corestN}
Under the hypotheses of Theorem \ref{main2}, it holds that
\begin{equation}
\label{corestNine}
 \lambda_{1}(\beta,\Omega)\le (1-p)|\beta|^{\frac{p}{p-1}}.
\end{equation}
\end{cor}
The estimate \eqref{corestNine} is known in the Euclidean case (see for example \cite{gs,S} for $p=2$ and $\cite{KP}$ for $1<p<+\infty$). The method of proof of \cite{gs,KP} is simpler and completely different from our method, and it seems that can be adapted to the Finsler case, but allows only to get a much worst estimate (see Remark 5.2 below).

The structure of the paper is the following. In the second section, we review some basic tools used in the paper as the Finsler norms, anisotropic curvatures and some properties of the first eigenvalue $\lambda_{1}(\beta,\Omega)$, as well as the $\mathcal P-$function method. In the third section, we discuss about a suitable one-dimensional $p$-Laplace eigenvalue problem. Finally, in the fourth section we prove two comparison results and in Section \ref{main_sec}, we give the proofs of theorems \ref{main1} and \ref{main2}, and of corollaries \ref{cor_H} and \ref{corestN}.

\section{Notation and preliminaries}\label{notations_sec}

Let $1<p<+\infty$; throughout this paper we will denote by $p'$ the conjugate exponent of $p$, that is the number such that $\frac 1p +\frac 1{p'}=1$. Moreover, we will call a domain of $\mathbb R^{n}$ a connected open set.

\subsection{The Finsler norm}
Throughout the paper we will consider a convex, even, $1-$homogeneous function 
\[
\xi\in \R^{N}\mapsto F(\xi)\in [0,+\infty[,
\] 
that is a convex function such that
\begin{equation}
\label{eq:omo}
F(t\xi)=|t|F(\xi), \quad t\in \R,\,\xi \in \R^{N}, 
\end{equation}
 and such that
\begin{equation}
\label{eq:lin}
a|\xi| \le F(\xi),\quad \xi \in \R^{N},
\end{equation}
for some constant $a>0$.  Under this hypothesis it is easy to see that there exists $b\ge a$ such that
\[
F(\xi)\le b |\xi|,\quad \xi \in \R^{N}.
\]
Throughout the paper we will also assume that $F$ belongs to $C^{3,\alpha}(\mathbb{R}^N\setminus \{0\})$, and that
\begin{equation}
\label{strong}
\nabla^{2}_{\xi}[F^{p}](\xi)\text{ is positive definite in }\R^{N}\setminus\{0\}.
\end{equation}

The hypothesis \eqref{strong} on $F$ ensures that the operator 
\[
\Qp[u]:= \dive \left(\frac{1}{p}\nabla_{\xi}[F^{p}](\nabla u)\right)
\] 
is elliptic, hence there exists a positive constant $\gamma$ such that
\begin{equation*}
\sum_{i,j=1}^{n}{\nabla^{2}_{\xi_{i}\xi_{j}}[F^{p}](\eta)
  \xi_i\xi_j}\ge
\gamma |\eta|^{p-2} |\xi|^2, 
\end{equation*}
for some positive constant $\gamma$, for any $\eta \in
\R^N\setminus\{0\}$ and for any $\xi\in \R^N$. 

For $p\ge 2$, the condition  
\begin{equation*}
\nabla^{2}_{\xi}[F^{2}](\xi)\text{ is positive definite in }\R^{N}\setminus\{0\},
\end{equation*}
is sufficient for the validity of \eqref{strong}.

The polar function $F^o\colon\R^N \rightarrow [0,+\infty[$ 
of $F$ is defined as
\begin{equation*}
F^o(v)=\sup_{\xi \ne 0} \frac{\langle \xi, v\rangle}{F(\xi)}. 
\end{equation*}
 It is easy to verify that also $F^o$ is a convex function
which satisfies properties \eqref{eq:omo} and
\eqref{eq:lin}. Furthermore, 
\begin{equation*}
F(v)=\sup_{\xi \ne 0} \frac{\langle \xi, v\rangle}{F^o(\xi)}.
\end{equation*}
From the above property it holds that
\begin{equation}\label{prodscal}
| \xi\cdot \eta | \le F(\xi) F^{o}(\eta) \qquad \forall \xi, \eta \in \R^{N}.
\end{equation}
The set
\[
\mathcal W = \{  \xi \in \R^N \colon F^o(\xi)< 1 \}
\]
is the so-called Wulff shape centered at the origin. We put
$\kappa_N=|\mathcal W|$, where $|\mathcal W|$ denotes the Lebesgue measure
of $\mathcal W$. More generally, we denote with $\mathcal W_r(x_0)$
the set $r\mathcal W+x_0$, that is the Wulff shape centered at $x_0$
with measure $\kappa_Nr^N$, and $\mathcal W_r(0)=\mathcal W_r$.

The following properties of $F$ and $F^o$ hold true:
\[
\begin{array}{ll}
 F_{\xi}(\xi) \cdot \xi = F(\xi), \quad  F_{\xi}^{o} (\xi)\cdot \xi 
= F^{o}(\xi), & \forall \xi \in
\R^N\setminus \{0\}, \\[.15cm]
 F(F_{\xi}^o(\xi))=F^o( F_{\xi}(\xi))=1,&\forall \xi \in
\R^N\setminus \{0\}, 
\\[.15cm]
F^o(\xi)  F_{\xi}(F_{\xi}^o(\xi) ) = F(\xi) 
F_{\xi}^o( F_{\xi}(\xi) ) = \xi & \forall \xi \in
\R^N\setminus \{0\},
\end{array}
\]
where $F_{\xi}=\nabla F(\xi)$. The anisotropic distance of $x\in\overline\Omega$ to the boundary of a bounded domain $\Omega$ is the function 
\begin{equation*}
d_{F}(x)= \inf_{y\in \de \Omega} F^o(x-y), \quad x\in \overline\Omega.
\end{equation*}

We stress that when $F=|\cdot|$ then $d_F=d_{\mathcal{E}}$, the Euclidean distance function from the boundary.

It is not difficult to prove that $d_{F}$ is a uniform Lipschitz function in $\overline \Omega$ and
\begin{equation*}
  F(\nabla d_F(x))=1 \quad\text{a.e. in }\Omega.
\end{equation*}
Obviously, $d_F\in W_{0}^{1,\infty}(\Omega)$. 
Finally, the anisotropic inradius of $\Omega$ is the quantity
\begin{equation*}
R_{F}(\Omega)=\max \{d_{F}(x),\; x\in\overline\Omega\},
\end{equation*}
that is the radius of the largest Wulff shape $\mathcal W_{r}(x)$ contained in $\Omega$.

\subsection{Anisotropic curvatures}
We recall here some basic facts on anisotropic curvatures and on an integration formula in anisotropic normal coordinates, referring the reader mainly to \cite{cm} for the details.

Let $\Omega$ be a $C^{2}$ bounded domain, and $x\in\de\Omega$.
The anisotropic outer normal $n^{F}_{\Omega}$ to $\de \Omega$ is given by
  \[
  n^F_{\Omega}(y)= F_{\xi}(\nu_{\Omega}(y)),\qquad y\in \de\Omega,
  \]
  where $\nu_{\Omega}(y)$ is the Euclidean outer normal to $\de \Omega$ at $y$.
It holds 
 \[
  F^o(n^F_{\Omega})=1.  
  \]
 The anisotropic principal curvatures $\kappa^{F}_{1},\ldots,\kappa_{N-1}^{F}$ are defined as the eigenvalues of the anisotropic Weingarten map
\[
d n^{F}_{\Omega}\colon T_{y}\de\Omega \to T_{n_{\Omega}^{F}(y)}\mathcal W,
\]
where $T_{y}\de\Omega$ is the tangent space to $\de\Omega$ at $y$ (see \cite{wx}). The anisotropic mean curvature of $\de\Omega$ at a point $y$ is defined as
  \begin{equation*}
  \mathcal H_{F}(y)= \kappa_{1}^{F}(y)+\ldots \kappa_{N-1}^{F}(y) , \quad y\in
  \de \Omega.
  \end{equation*}
A domain such that $\mathcal H_{F}(y)\ge 0$ for any $y\in \de\Omega$ is said to be anisotropic mean convex. 
   
Actually, the anisotropic principal curvatures can be equivalently defined as follows (\cite{cm}, in particular  Remark 5.9). Let us recall that the anisotropic distance function $d_{F}$ is $C^{2}$ in a tubular neighborhood of $\de\Omega$, hence we can define the matrix-valued function
\[
W(y)=-F_{\xi\xi}(\nabla d_{\Omega}(y))\nabla^{2}d_{\Omega}(y), \quad y\in\de\Omega.
\]
It holds that for any $v\in\R^{N}$, $W(y)v$ belongs to $T_{y}$. Then it is possible to define the map $\overline W(y)\colon T_{y}\to T_{y}$,  as $\overline W(y)w=W(y)w$, $w\in T_{y}$. 
The matrix $\overline W(y)$ is, in general, not symmetric, but the eigenvalues are real numbers. These eigenvalues $\kappa_{1}^{F}(y)\le \kappa_{2}^{F}(y)\le \ldots\le \kappa_{N-1}^{F}(y)$ are the anisotropic principal curvatures of $\de \Omega$ at the point $y$, and the definition is equivalent to the preceding one. Then it holds that
\[  
  \mathcal H_{F}(y)= 
  \dive\left[ 
    F_{\xi}\left(-{\nabla d_{F}(y)}\right) \right] = \tr(W(y)).
\]
(see also \cite[Section 3]{wx}).

One of the tools that will be used in what follows, is a change of variable formula in anisotropic normal coordinates, based on the fact that the set of the singular points of the anisotropic distance function from the boundary has Lebesgue measure zero (\cite{cm}). To state this result, let us define the function
\[
\Phi(y,t)= y-tF_{\xi}(\nu_{\Omega}(y)),\qquad y\in \de\Omega, \quad t\in \R,
\]
and the function $\ell(y)$, $y\in \de\Omega$, 
\[
\ell(y)=\sup\{d_{F}(z),\, z\in\Omega\text{ and }y\in \Pi(z) )\},
\]
where $\Pi(z)=\{\eta\in \de\Omega\colon d_{F}(z)=F^{o}(z-\eta)\}$ is the set of projections of $z$ in $\de\Omega$.
\begin{thm} \cite[Theorem 7.1]{cm}.\label{change_normal}
For every $h\in L^1(\Omega)$, it holds
\[
\int_\Omega h(x)dx=\int_{\partial\Omega} F(\nu_\Omega(y))\int_0^{\ell(y)}h(\Phi(y,t))J(y,t)\,dt\,d\mathcal H^{N-1}(y),
\]
where $J(y,t)=\Pi_{i=1}^{N-1}(1- t \kappa^{F}_i(y))$.
\end{thm}
The Jacobian of $\Phi$, that is $J(y,t)=\Pi_{i=1}^{N-1}(1- t \kappa^{F}_i(y))$, is positive, being $1- t \kappa^{F}_i(y)>0$ for any $i=1,\ldots,N-1$ (\cite[Lemma 5.4]{cm}), and if $\Omega$ is anisotropic mean convex it holds that
\begin{equation}
\label{laplaciano_cambio}
\frac{-\frac{d}{dt}\left[J(y,t)\right]}{J(y,t)}=\sum_{i=1}^{N-1}\frac{\kappa^{F}_i(y)}{1-t\kappa^{F}_i(y)}\ge \frac{\mathcal H_F(y)}{1-\frac{\mathcal H_F(y)}{N-1} t}\geq 0\qquad\forall y\in\partial\Omega,\; t\in[0,\ell (y)[.
\end{equation}
The equality in \eqref{laplaciano_cambio} is a straightforward computation, while the first inequality follows by applying the Newton inequalities exactly as done in  \cite[Proposition 2.6]{LLL}.

To conclude this subsection, we recall an observation that will be used in the following. For any $x\in \Omega$ such that $\Pi(x)=\{y\}$, we have (\cite[Lemma 4.3]{cm}) that
\begin{equation}
\label{dentro_fuori}
\nabla d_F(x)=-\frac{\nu_\Omega(y)}{F(\nu_{\Omega}(y))}.
\end{equation}

\subsection{The first eigenvalue of the Robin Finsler $p$-Laplacian}
In this section we review an existence and characterization result for the first eigenpair of the Robin Finsler $p-$Laplacian. Recall that, in the case $\beta>0$, the existence result has been proved for example in \cite{pota}; for the negative values of the boundary parameter, it follows by adapting the proof of \cite{KP} for the Euclidean $p$-Laplacian.

\begin{prop}
Let $1<p<+\infty$, $\beta\in\R$, and let $\Omega$ be a bounded Lipschitz domain in $\R^{N}$. Then there exists a function $u\in C^{1,\alpha}(\Omega)$ which realizes the minimum in \eqref{eigint} and satisfies problem \eqref{pb_Robin1}.  
Moreover, the first eigenvalue $\lambda_1(\beta,\Omega)$ of \eqref{pb_Robin1} is simple and the corresponding eigenfunctions are positive (or negative) in $\Omega$.
\end{prop}
\begin{proof}
{For any $\beta\in\R$, let us consider $\{u_n\}_{n\in\N}$ a minimizing normalized sequence for $\lambda_1(\beta,\Omega)$, that is
\begin{multline}
\label{min_seq}
\lim_{n\to+\infty} \int_\Omega F(\nabla u_n)^p dx +\beta \ds\int_{\de\Omega}|u_n|^pF(\nu_\Omega) d\cH^{N-1} = \lambda_1(\beta,\Omega), \\ \text{with}\int_\Omega |u_n|^p dx=1 \ \forall n\in\N.
\end{multline}
For any $\varepsilon\in(0,1)$, there exists a constant $K_\varepsilon>0$ (see \cite[Th. 1.5.1.10]{G}) such that
\begin{equation}
\label{grisvard_est}
\int_{\partial \Omega}|u_n|^pdx\le\varepsilon \int_{\Omega}|\nabla u_n|^pdx+K_\varepsilon\int_\Omega |u_n|^pdx\quad\forall n\in\N.
\end{equation}
By using \eqref{eq:lin}, \eqref{min_seq} and \eqref{grisvard_est} for suitable choice of $\varepsilon$ depending on $\beta$, we deduce that
\[
\sup_{n\in\N} \int_\Omega |\nabla u_n|^p dx <+\infty.
\]
Therefore $\{u_n\}_{n\in\N}$ is bounded in $W^{1,p}(\Omega)$ and hence there exists a subsequence $\{u_{n_j}\}$ and $u\in W^{1,p}(\Omega)$ such that $u_{n_j}\to u$ strongly in $L^p(\Omega)$ and almost everywhere, and $\nabla u_{n_j}\rightharpoonup \nabla u$ weakly in $L^p(\Omega)$. As a consequence, by the compactness of the
trace operator (see for example \cite[Cor. 18.4]{leoni}), $u_{n_j}$ (up to another subsequence) strongly converges to $u$ in $L^p(\de \Omega)$.}

Hence, the existence of the minimizer of \eqref{eigint} follows by the lower semicontinuity:
\[
\begin{split}
\lambda_1(\beta,\Omega)&\leq \int_\Omega F(\nabla u)^p dx +\beta \ds\int_{\de\Omega}|u|^pF(\nu_\Omega) d\cH^{N-1}\\&
\leq \liminf_{j\to+\infty} \int_\Omega F(\nabla u_{n_j})^p dx +\beta \ds\int_{\de\Omega}|u_{n_j}|^pF(\nu_\Omega) d\cH^{N-1}=\lambda_1(\beta,\Omega).
\end{split}
\]

Now, let us observe that if $u$ is a minimizer of \eqref{eigint}, then $|u|$ is also a minimizer, so we can assume $u\ge 0$. Moreover, $u$ is in $L^{\infty}(\bar\Omega)$ (see \cite{pota} if $\beta\ge 0$, or argue as in \cite[Theorem 4.3]{lea} if $\beta<0$). Hence by Harnack inequality, $u>0$ in $\Omega$. Moreover, by standard regularity results (see for example \cite{to}), $u\in C^{1,\alpha}(\Omega)$. Finally, to prove the simplicity we follow an argument contained in \cite{BK}. So, let fix two positive eigenfunctions $u$ and $v$, normalized as $\|u\|_{L^{p}(\Omega)}=\|v\|_{L^{p}(\Omega)}=1$, and define $w_t=\left( t u^p+(1-t)v^p  \right)^{1/p}$, with $t\in[0,1]$. We have that $\|w_{t}\|_{L^{p}(\Omega)}=1$. Then, by the convexity of $F^p$, it holds
\begin{equation}\label{eqconc}\begin{split}
F(\nabla w_t)^p & =w_t^p F\left(\frac{t u^p \frac{\nabla u}u+(1-t)v^p\frac{\nabla v}{v}}{tu^p+(1-t)v^p}\right)^p\\
 & \leq \mu_t^p\left[\frac{t u^p}{tu^p+(1-t)v^p}F\left(\frac{\nabla u}u\right)^p+\frac{(1-t)v^p}{tu^p+(1-t)v^p}F\left(\frac{\nabla v}v\right)^p\right]\\&=t F(\nabla u)^p+(1-t)F(\nabla v)^p.
\end{split}
\end{equation}
By using \eqref{eqconc} and the fact that $u$ and $v$ are both minimizers, we obtain
\begin{equation*}
\begin{split}
\lambda_1(\beta,\Omega)&\leq \int_\Omega F(\nabla w_t)^p dx +\beta \ds\int_{\de\Omega} w_t^pF(\nu_\Omega) d\cH^{N-1}\\
&\leq t\left(\int_\Omega F(\nabla u)^p dx +\beta \ds\int_{\de\Omega}u^pF(\nu_\Omega) d\cH^{N-1}\right)\\
&+ (1-t)\left(\int_\Omega F(\nabla v)^p dx +\beta \ds\int_{\de\Omega}v^pF(\nu_\Omega) d\cH^{N-1}\right)\\
&=\lambda_1(\beta,\Omega)
\end{split}
\end{equation*}
So $w_t$ is a minimizer for \eqref{eigint} and the inequality \eqref{eqconc} holds as equality. Therefore $\frac{\nabla u}u=\frac{\nabla v}v$ and, consequently, $\nabla (\log u-\log v)=0$, that implies the result.
\end{proof}

\subsection{The $\mathcal P$-function method}
In order to give an optimal lower estimate for  $\lambda_{1}(\beta,\Omega)$
in the case $\beta>0$, we will use the so-called $\mathcal P$-function method.
Let us consider the Dirichlet boundary value problem 
\begin{equation}
\label{pb_gen}
\begin{cases}
-\mathcal Q_{p}w =f(w) &\text{in }\Omega\\
w=0 &\text{on }\de\Omega,
\end{cases}
\end{equation}
where $f$ is a nonnegative $C^{1}(0,+\infty)\cap C^{0}([0,+\infty[)$ function. We recall the following result, based on \cite[Prop. 4.1]{CFV}.

\begin{thm}[\cite{DPGGLB,DPdBG}]
\label{prsperb}
Let $\Omega$ be a bounded $C^{2}$ domain in $ \R^{N}$, $ N \ge 2$, with nonnegative anisotropic mean curvature $\mathcal H_{F}$ of $\de\Omega$, and let $w>0$ be a solution to the problem  \eqref{pb_gen}, 
then
\begin{equation*}
\mathcal P(x):=\frac{p-1}{p}F^p(\nabla w(x))-\int_{w(x)}^{\max_{\bar \Omega} w}f(s)ds \le 0  \quad \text{ in } \overline{\Omega}.
\end{equation*}
\end{thm}

\begin{rem}
\label{rem_P_f} Theorem \ref{prsperb} implies that if $f=1$ and then $w=w_\Omega\in W_0^{1,p}(\Omega)$ is the stress function for $\Omega$, that is the solution of \eqref{pbtord} then
\begin{equation}
\label{pfunctss}
F(\nabla w_\Omega(x))\le \left[\frac{p}{p-1}(M(\Omega) -w_\Omega(x)) \right]^{\frac{1}{p}}
\end{equation}
where $M(\Omega)=\max_\Omega w_\Omega$.
\end{rem}

\section{A one dimensional $p$-Laplacian eigenvalue problem}\label{1D_subsec}
In this Section, we study a one-dimensional eigenvalue problem that will be useful to estimate $\lambda_1(\beta,\Omega)$ for both positive and negative values of the boundary parameter.

We first briefly summarize the definitions and some properties of the $p$-trigonometric and $p$-hyperbolic functions (see for example to \cite{LE,L,pee}).  We observe that these functions coincide with the standard trigonometric and hyperbolic functions when $p=2$.

Let us consider the function $\arccos_p:[0, 1]\to\R$ defined as
\begin{equation*}
\arccos_{p}(x)=\ds\int_x^1\frac{dt}{\ds\left(1-{t^p}\right)^\frac 1p}.
\end{equation*}
Denote by $z(t)$ the inverse function of $\arccos_{p}$ which is defined on the interval $\left[0,\frac{\pi_p}{2}\right]$, where
\[
\pi_p=2\ds\int_{ 0}^{1}\frac{dt}{\ds\left(1-{t^p}\right)^\frac 1p}
= \frac{2\pi}{p\sin\frac{\pi}{p}}
.
\]
Therefore, the $p$-cosine function $\cos_p$ is the even function defined as the following periodic extension of $z(t)$:
\[
\cos_p(t)=\left\{ \begin{split}
& z(t) &&\text{if}\ \ t\in\left[0,\frac{\pi_p}{2}\right],\\
& -z(\pi_p-t) &&\text{if}\ \ t\in\left[\frac{\pi_p}{2}, \pi_p\right], \\
& \cos_p(-t)\  &&\text{if}\ \ t\in\left[-\pi_p, 0\right], \\
\end{split}
\right.
\]
and extended periodically to all $\R$, with period $2\pi_p$; the extension is continuosly differentiable on $\R$. If $p=2$, then $\cos_{p}x$ and $\arccos_{p}x$ coincides with the standard trigonometric functions $\cos x$ and $\arccos x$.

Let now also recall what the generalized hyperbolic cosine and arccosine functions are. The function $\arccosh_{p}$ is defined as
\begin{equation*}
\arccosh_p(x)=\ds\int_1^x\frac{1}{(t^p-1)^\frac1p}dt,\ x\in [1,+\infty[.
\end{equation*}
Then $\cosh_{p} \colon t\in [0,+\infty[\mapsto[1,+\infty[$ is its inverse function. It is strictly increasing in $[0,+\infty[$; it can be extended on all $\R$ as $\cosh_{p}(-t)=\cosh_{p}(t)$, $t>0$.
If $p=2$, $\cosh_{p}$ and $\arccosh_{p}$ coincides with the standard hyperbolic functions.

Now we consider the following eigenvalue problem in the unknown $X=X(s)$:
\begin{equation}
\label{1dim}
    \begin{cases}
    (|X'|^{p-2}X')'+\mu |X|^{p-2}X=0 \quad\text{in}\ (0,s_0),\\
    |X'(0)|^{p-2}X'(0)=0, \\
    |X'(s_{0})|^{p-2}X'(s_{0})+\beta |X(s_{0})|^{p-2}X(s_{0})=0, 
    \end{cases}
\end{equation}
where $s_{0}$ is a given positive number.

\begin{thm}\label{mu_thm}
Let $1<p<+\infty$ and $\beta \in \R$. Then there exists the smallest  eigenvalue $\mu$ of  \eqref{1dim}, which has the following variational characterization:
\[
\mu_{1}(\beta,s_0)=\inf_{\substack{v\in W^{1,p}(0,s_{0}) \\ \left|v'(0)\right|^{p-2}v'(0)=0}} \frac{\int_{0}^{s_{0}}\left|v'(s)\right|^{p}ds+\beta v(s_{0})^{p}}{\int_{0}^{s_{0}}\left|v(s)\right|^{p}ds}.
\]
 Moreover, the corresponding eigenfunctions are unique up to a multiplicative constant and have constant sign. The first eigenvalue $\mu_{1}(\beta,s_0)$ has the sign of $\beta$. Moreover if $\beta>0$, the first eigenfunction is
\begin{equation*}
X(s)=\cos_p\left(\left(\frac{\mu_{1}(\beta,s_0)}{p-1}\right)^\frac{1}{p}s \right),\quad s\in (0,s_{0})
\end{equation*}
and the eigenvalue $\mu_{1}(\beta,s_0)$ is the first positive value that satisfies
\begin{equation*}
\frac{\mu}{p-1}=\frac{\beta^{p'}}{\cos_p^{-p}\left(\left(\frac{\mu}{p-1}\right)^\frac 1p s_{0} \right)-1}.
\end{equation*}

If $\beta<0$, the first eigenfunction is
\begin{equation*}
X(s)=\cosh_p\left(\left(\frac{-\mu_{1}(\beta,s_0)}{p-1}\right)^\frac{1}{p}s \right),\quad s\in (0,s_{0})
\end{equation*}
and in this case $\mu_{1}(\beta,s_0)$ is the unique negative value that satisfies
\begin{equation*}
-\frac{\mu}{p-1} =\frac{|\beta|^{p'}}{1-\cosh_p^{-p}\left(\left(\frac{-\mu}{p-1}\right)^\frac 1p s_{0} \right)}.
\end{equation*}

\end{thm}
\begin{proof}
By standard arguments of Calculus of Variations, it is possible to show the existence of the first eigenvalue $\mu_{1}(\beta,s_0)$ in correspondence with a unique (up to a constant) positive solution of \eqref{1dim}, that we denote by $X$. Moreover, $\mu_{1}(\beta,s_0)$ has the sign of $\beta$, and integrating \eqref{1dim} we get 
\begin{equation}
\label{mon}
|X'(s)|^{p-2}X'(s)=-\mu_{1}(\beta,s_0) \int_{0}^{s}X^{p-1}dt.
\end{equation}
Moreover, by multiplying the first line of \eqref{1dim} for $X'$ and integrating from $0$ to $s$, we have
\[
\int_0^{s}(|X'|^{p-2}X')'X' dt+\mu_{1}(\beta,s_0) \int_{0}^{s}X^{p-1}X'dt=0
\]
and hence 
\begin{equation}\label{equa_0}
\frac{|X'(s)|^p}{p'}+\frac{\mu_{1}(\beta,s_{0})}{p} X(s)^p=\frac{\mu_{1}(\beta,s_{0})}{p} X(0)^p\qquad \text{in}\ (0,s_0).
\end{equation}

Now, let $\beta>0$. By \eqref{mon}, $X$ is decreasing; we fix $X(0)=1$. Therefore, from \eqref{equa_0}, we have
\begin{equation}\label{eq_diff}
X'=-\tilde\mu^\frac 1p(1-X^p)^\frac{1}{p},
\end{equation}
where $\tilde\mu= \frac{\mu_{1}(\beta,s_0)}{p-1} $. By integrating \eqref{eq_diff} from $0$ to $s$, we have
\begin{equation}
\label{arcsine}
\arccos_p\left(X(s)\right)=\ds\int_{X(s)}^{1}\frac{1}{(1- X^p)^\frac1p}dX=\tilde\mu^\frac 1ps
\end{equation}
Then the solution of problem \eqref{1dim} is
\begin{equation}
\label{sol_0}
X(s)=\cos_p\left(\tilde\mu^\frac{1}{p}s \right).
\end{equation}

Finally, from the boundary condition in \eqref{1dim}, we have $X'(s_{0})=-\beta ^\frac{1}{p-1}X(s_{0})$, that, by  considering \eqref{eq_diff} and \eqref{sol_0}, gives
\begin{equation*}
\frac{\tilde\mu}{\beta^{p'}}=\frac{1}{\cos_p^{-p}\left(\tilde\mu ^\frac 1p s_{0} \right)-1}.
\end{equation*}
 
When $\beta<0$, then $X>0$ is increasing and defined up to a multiplicative constant, we set the minimum $X(0)=1$. From \eqref{equa_0}, we have
\begin{equation}\label{eq_diffN}
X'=(-\tilde\mu)^{\frac 1p}(X^p-1)^\frac 1p.
\end{equation}
By integrating from $0$ to $s$, we have
\[
\arccosh_p\left(X(s)\right)=\ds\int_1^{X(s)}\frac{1}{(X^p- 1)^\frac 1p}dX=(-\tilde\mu)^\frac 1ps.
\]
Therefore, the solution of \eqref{1dim}, for $\beta<0$, is
\begin{equation}
\label{sol_0N}
X(s)=\cosh_p\left((-\tilde\mu)^\frac 1ps \right).
\end{equation}
Moreover, from the boundary condition in \eqref{1dim}, we have $X'({s_0})=|\beta|^\frac{1}{p-1}X({s_0})$ that, by  considering \eqref{eq_diffN} and \eqref{sol_0N}, gives
\begin{equation} 
\label{tildemuneg}
\frac{-\tilde\mu}{|\beta|^{p'}} =\frac{1}{1-\cosh_p^{-p}\left((-\tilde\mu)^\frac 1p s_{0} \right)}.
\end{equation}
We observe that, being the function $\cosh_{p}(t)$   strictly monotone increasing as $t>0$, and $\cosh_{p}(0)=1$, $\cosh_{p}(t)\to +\infty$ as $t\to +\infty$, then 
$\frac{x}{|\beta|^{p'}}\left[ 1-\cosh_p^{-p}\left(x^\frac 1p s_{0} \right)\right] =1$ has always a unique positive solution $x$.
\end{proof}

\begin{rem}\label{argument}
Since $X$ is decreasing when $\beta>0$, it holds that $0\leq X\leq 1$. Hence $0\leq \arccos_p\left(X(s)\right)\leq \frac{\pi_p}{2}$. Therefore, from \eqref{arcsine}, we have $0\leq \tilde\mu^\frac 1p s \leq\frac{\pi_p}{2}$, and then
\[
\mu_{1}(\beta,s_0) \leq (p-1)\left(\frac{\pi_p}{2s_{0}}\right)^{p}.
\] 

On the other hand, in the case $\beta<0$, then by \eqref{tildemuneg} it holds 
\[
\mu_{1}(\beta,s_0) \le -(p-1)|\beta|^{p'}.
\]
\end{rem}

\section{Comparison results}
In this section we will prove a comparison result that will be used to get the main theorems. We remark that the two estimates stated below do not depend on the sign of $\beta$.
\begin{prop}
\label{usefullemma}
Let $\Omega$ be a bounded Lipschitz domain in $ \R^{N}$, $ N \ge 2$, $\beta\in\R$ and $\lambda_{1}(\beta,\Omega)$ be the first eigenvalue of \eqref{pb_Robin1}.

If $v$ is a nonnegative function satisfying
\begin{equation}
    \label{pb_diseq}
\begin{cases}
-\mathcal Q_p v \geq \mu v^{p-1} \quad &\text{in}\ \Omega\\
F(\nabla v)^{p-1} F_{\xi}(\nabla v)\cdot \nu_\Omega +\beta F(\nu_\Omega) v^{p-1}\ge 0 & \text{on}\ \partial\Omega,
\end{cases}
\end{equation}
in weak sense, that is
\begin{multline*}
\int_{\Omega}F(\nabla v)^{p-1}F_{\xi}(\nabla v) \cdot \nabla \varphi dx+\beta\int_{\Omega} F(\nu_\Omega) v^{p-1} d\cH^{N-1} \ge \mu \int_{\Omega} v^{p-1}\varphi dx, \\
\forall \varphi \in W^{1,p}(\Omega),\; \varphi \ge 0,
\end{multline*}
then
\[
\lambda_{1}(\beta,\Omega)\ge \mu.
\]

If 
$v$ is a nonnegative function satisfying
\begin{equation}
    \label{pb_diseqN}
\begin{cases}
-\mathcal Q_p v \leq \mu v^{p-1} \quad &\text{in}\ \Omega\\
F(\nabla v)^{p-1} F_{\xi}(\nabla v)\cdot \nu_\Omega +\beta F(\nu_\Omega) v^{p-1}\le 0 & \text{on}\ \partial\Omega,
\end{cases}
\end{equation}
in weak sense, that is
\begin{multline*}
\int_{\Omega}F(\nabla v)^{p-1}F_{\xi}(\nabla v) \cdot \nabla \varphi dx+\beta\int_{\Omega} F(\nu_\Omega) v^{p-1} d\cH^{N-1} \le \mu \int_{\Omega} v^{p-1}\varphi dx, \\
\forall \varphi \in W^{1,p}(\Omega),\; \varphi \ge 0,
\end{multline*}
Then
\[
\lambda_{1}(\beta,\Omega)\le \mu.
\]
\end{prop}
 \begin{proof}
Let us consider a positive first eigenfunction $u$ of \eqref{pb_Robin1}. Hence
\begin{equation}
\label{weak_u}
\int_\Omega F(\nabla u)^{p} dx + \beta \int_{\partial\Omega} u^{p}F(\nu_\Omega)d\cH^{N-1}= \lambda_{1}(\beta,\Omega) \int_{\Omega}u^{p}dx
\end{equation}
while, choosing as test function $\varphi =\frac{u^p}{(v+\varepsilon)^{p-1}}$ in \eqref{pb_diseq} we get
 \begin{equation}
     \label{weak_v}
\begin{split}
\int_\Omega p\left(\frac{u}{v+\varepsilon}\right)^{p-1} F(\nabla v)^{p-1}  F_{\xi} (\nabla v) \cdot \nabla u\, dx - (p-1)\int_\Omega\left(\frac{u}{v+\varepsilon}\right)^p F(\nabla v)^pdx 
\\
+\beta\int_{\partial\Omega}  \frac{v^{p-1}}{(v+\varepsilon)^{p-1}} u^{p}F(\nu_\Omega)d\cH^{N-1}
\ge
\mu\int_\Omega\frac{v^{p-1}}{(v+\varepsilon)^{p-1}}
u^{p} dx
\end{split}
 \end{equation}
Hence by subtracting \eqref{weak_v} to \eqref{weak_u}, we get
	\begin{multline*}
	\int_{\Omega} \left\{[ F(\nabla u)]^{p} - p 
	F\left(\frac{u}{v+\eps}\nabla v\right)^{p-1}
	F_{\xi}\left(\frac{u}{v+\eps}\nabla v\right)  \cdot \nabla u   +\right.\\ \left. + (p-1) F\Big( \frac{u}{v+\eps}  \nabla v\Big)^{p} 
	 \right\}dx \le  
	 \int_{\Omega} \left[\lambda_{1}(\beta,\Omega)-\frac{v^{p-1}}{(v+\eps)^{{p-1}}}\mu\right] u^{p}\, dx.
	\end{multline*}
	From the convexity of $F^{p}$, we get that the left-hand side in the above inequality is nonnegative. Hence, as $\eps \rightarrow 0$, the monotone convergence gives that
	\[
	0\le (\lambda_{1}(\beta,\Omega)-\mu) \int_{\Omega} u^{p}\, dx,
	\]
	 and the first inequality follows.

As regards the second inequality, the proof follows in a similar way: by using $\frac{v^p}{(u+\varepsilon)^{p-1}}$ and $v$ as test functions in \eqref{pb_Robin1} and \eqref{pb_diseqN} respectively, we get
 \begin{equation}
 \label{weak_uN}
\begin{split}
\int_\Omega p\left(\frac{v}{u+\varepsilon}\right)^{p-1} F(\nabla u)^{p-1}  F_{\xi} (\nabla u) \cdot \nabla v\, dx - (p-1)\int_\Omega\left(\frac{v}{u+\varepsilon}\right)^p F(\nabla u)^pdx\\
+\beta\int_{\partial\Omega}  \frac{u^{p-1}}{(u+\varepsilon)^{p-1}} v^{p}F(\nu_\Omega)d\cH^{N-1}=\lambda_1(\beta,\Omega)\int_\Omega\frac{u^{p-1}}{(u+\varepsilon)^{p-1}}v^{p} dx.
\end{split}
 \end{equation}
 and
\begin{equation}
\label{weak_vN}
\int_\Omega F(\nabla v)^{p} dx + \beta \int_{\partial\Omega} v^{p}F(\nu_\Omega)d\cH^{N-1}\le \mu \int_{\Omega}v^{p}dx;
\end{equation}
hence by subtracting \eqref{weak_uN} to \eqref{weak_vN}  we get 
\begin{multline*}
\int_{\Omega} \left\{[ F(\nabla v)]^{p} - pF\left(\frac{v}{u+\eps}\nabla u\right)^{p-1}F_{\xi}\left(\frac{v}{u+\eps}\nabla u\right)  \cdot \nabla v   +\right.\\ 
\left. + (p-1) F\Big( \frac{v}{u+\eps}  \nabla u\Big)^{p} \right\}dx \le  
	 \int_{\Omega} \left[\mu-\frac{u^{p-1}}{(u+\eps)^{{p-1}}}\lambda_{1}(\beta,\Omega)\right] v^{p}\, dx.
	\end{multline*}
and passing to the limit as $\eps\to 0$ 
	\[
	(\mu-\lambda_{1}(\beta,\Omega)) \int_{\Omega} v^{p}\, dx\geq 0,
	\]
	 and the conclusion follows.
\end{proof}

\section{Proof of the main results}\label{main_sec}
In this section we give the proof of the main results of the paper. 
\begin{proof}[Proof of Theorem \ref{main1}] Let us consider the anisotropic Dirichlet $p$-torsional rigidity problem \eqref{pbtord} and let us denote by $M$ the maximum achieved by the anisotropic stress function $w_{\Omega}$, and fix $s_0=(p'M)^{\frac{1}{p'}}$. The first eigenfunction $X=X(s)$ of problem \eqref{1dim} with $\mu=\mu_{1}(\beta,s_0)$ has a sign, then we can assume that $X>0$ (see also the discussion in Section \ref{1D_subsec}). We set
\[
s(x):=\left[p'(M-w_{\Omega}(x))\right]^{\frac{1}{p'}}\qquad\text{in}\ \Omega,
 \]
and let us observe that $s(x)=s_{0}$ if $x\in\de\Omega$.
We denote
\[
v(x):=X(s(x))\qquad\text{in}\ \Omega.
\]
Therefore, we have $\nabla v(x)=X'(s(x))\nabla s(x)$, and
\[
\nabla s=-\frac{\nabla w_{\Omega}}{s^{p'-1}}\qquad\text{in}\ \Omega.
\]
We first show that the function $v$ satisfies
\begin{equation*}
\begin{cases}
-\mathcal Q_p v \geq \mu_{1}(\beta,s_0) v^{p-1} \quad &\text{in}\ \Omega\\
F(\nabla v)^{p-1} F_{\xi}(\nabla v)\cdot \nu_\Omega +\beta F(\nu_\Omega) v^{p-1}\ge 0 & \text{on}\ \partial\Omega.
\end{cases}
\end{equation*}
Indeed, the computation of the anisotropic $p$-Laplacian of $v$ gives:
\[
\begin{split}
\mathcal Q_{p}v&=\dive ( F^{p-1}(\nabla v) F_{\xi}(\nabla v))= \dive (|X'|^{p-2}X'F^{p-1}(\nabla s)F_\xi(\nabla s))\\
&=|X'|^{p-2}X'\dive (F^{p-1}(\nabla s)F_\xi(\nabla s))+ (|X'|^{p-2}X')'F^{p}(\nabla s)\\
&=|X'|^{p-2}X'\dive \left(-\frac{F^{p-1}(\nabla w_{\Omega})F_{\xi}(\nabla w_{\Omega})}{s}\right)+ (|X'|^{p-2}X')'\frac{F^{p}(\nabla w_{\Omega})}{s^{p'}}\\
&=|X'|^{p-2}X' \left(-\frac{\mathcal Q_p w_{\Omega}}s+\frac{F_\xi(\nabla w_{\Omega})F^{p-1}(\nabla w_{\Omega})\nabla s}{s^2}\right)-\mu_{1}(\beta,s_0)X^{p-1}\frac{F^{p}(\nabla w_{\Omega})}{s^{p'}}\\
&=|X'|^{p-2}X' \left(\frac{1}s-\frac{F^{p}(\nabla w_{\Omega})}{s^{p'+1}}\right)-\mu_{1}(\beta,s_0) X^{p-1}\frac{F^{p}(\nabla w_{\Omega})}{s^{p'}}.
\end{split}
\]

Hence, we have the following:
\begin{equation}
\label{eq111}
\begin{split}
\mathcal Q_{p}v+\mu_{1} v^{p-1}&=\frac{|X'|^{p-2}X'}s \left(1-\frac{F^{p}(\nabla w_{\Omega})}{s^{p'}}\right)+\mu_{1}(\beta,s_0) X^{p-1}\left(1-\frac{F^{p}(\nabla w_{\Omega})}{s^{p'}}\right)\\
&=\left(\frac{|X'|^{p-2}X'}s+\mu_{1}(\beta,s_0) X^{p-1}\right)\left(1-\frac{F^{p}(\nabla w_{\Omega})}{s^{p'}}\right).
\end{split}
\end{equation}
We claim that
\[
\mathcal Q_{p}v+\mu_{1}(\beta,s_0) v^{p-1} \le 0.
\]
Indeed, the $\mathcal P$-function method assures that the term $1-\frac{F^{p}(\nabla w_{\Omega})}{s^{p'}}$ in \eqref{eq111} is nonnegative (see Remark \ref{rem_P_f});  hence we have to study the sign of
\begin{equation*}
g(s):=|X'|^{p-2}X'+\mu_{1}  \,s\, X^{p-1}.
\end{equation*}
Since $X$ is decreasing, by deriving $g$ with respect to $s$,  we have
\[
g'(s)=\mu_{1}(\beta,s_0) s (p-1) |X|^{p-2}X'\le 0,
\]
since $\mu_{1}(\beta,s_0)>0$ and $X$ is decreasing. Being $g(0)=0$, the claim is proved.

It remains to verify the boundary condition. Therefore on $\partial\Omega$, we have
\[
\begin{split}
F(\nabla v)^{p-1} F_\xi(\nabla v)\cdot \nu_\Omega &+\beta |v|^{p-2}v F(\nu_\Omega) \\ &=
-(-X'(s_{0}))^{p-1} F(\nabla s)^{p-1} F_\xi(\nabla s)\cdot \nu_\Omega +\beta X(s_{0})^{p-1}F(\nu_\Omega) \\
&=(-X'(s_{0}))^{p-1}\frac{F(\nabla w_\Omega)^{p-1} F_\xi(\nabla  w_\Omega)\cdot \nu_\Omega}{s_0} +\beta X(s_{0})^{p-1}F(\nu_\Omega) \\
&=(-X'(s_{0}))^{p-1}\frac{F(\nabla w_\Omega)^{p-1} F_\xi(\nabla  w_\Omega)\cdot \nu_\Omega}{s_0} +\beta X(s_{0})^{p-1}F(\nu_\Omega)\\
&=\beta X(s_{0})^{p-1}\left(F(\nu_\Omega)+\frac{F(\nabla w_\Omega)^{p-1} F_\xi(\nabla  w_\Omega)\cdot \nu_\Omega }{s_0}\right)\\
&=\beta X(s_{0})^{p-1}\frac{F(\nabla w_{\Omega})}{\left|\nabla w_{\Omega}\right|}\left(1-\frac{F(\nabla w_\Omega)^{p-1}}{s_0}\right).
\end{split}
\]
The final equality holds being $\nu_\Omega=-\frac{\nabla w_{\Omega}}{\left|w_{\Omega}\right|}$ on $\de\Omega$ and by the homogeneity of $F$, while the inequality is again a consequence of  \eqref{pfunctss}. 
The proof of \eqref{sperbine} is then completed by applying Lemma \ref{usefullemma}.
\end{proof}

\begin{proof}[Proof of Corollary \ref{cor_H}]
On one hand, we know from \eqref{eq_diff} that $X'(s_0)=-\tilde\mu^\frac 1p(1-X(s_0)^p)^{{\frac1p}}$, where $\tilde\mu=\frac{\mu_{1}(\beta,s_0)}{p-1}$; on the other hand, from the boundary condition of problem \eqref{1dim} in $s_0$, we have $X'(s_0)=-\beta^\frac{1}{p-1}X(s_0)$. Hence, this two equalities gives
\begin{equation}
\label{two_equalities}
X(s_0)=\left(\frac{\tilde\mu}{\tilde\mu+\beta^{p'}}\right)^\frac1p.
\end{equation}
By substituting the solution \eqref{sol_0} of problem \eqref{1dim} in \eqref{two_equalities}, we have
\[
\cos_p\left(s_0\tilde\mu^\frac 1p \right)=\left(\frac{\tilde\mu}{\tilde\mu+\beta^{p'}}\right)^\frac1p,
\]
and hence
\begin{equation*}
s_0\tilde\mu^\frac 1p=\arccos_{p}\left[ \left(\frac{\tilde\mu}{\tilde\mu+\beta^{p'}}\right)^\frac1p\right],
\end{equation*}
By recalling (\cite[Th. 1]{BV}) that
\[
\frac{\pi_{p}}{2}-\arccos_{p} x< \frac{\pi_{p}}{2}x,\quad x\in]0,1[,
\]
we have
\[
\frac{\pi_p}{2}-s_0\tilde\mu^\frac 1p\le \frac{\pi_p}{2}  \frac{\tilde\mu^\frac 1p}{\beta^\frac 1{p-1}}.
\]
By substituting $s_0=(p'M(\Omega))^\frac{1}{p'}$ in the latter inequality and then using that, by Theorem \ref{paynemax}, $p' M(\Omega)\leq R_F(\Omega)^{p'}$, it holds that
\[
\mu(\beta,s_{0}) \ge (p-1)\left( \frac{\pi_p}{2}\right)^p\frac{1}{\left(R_F(\Omega)+\frac{\pi_p}{2}\beta^{-\frac1{p-1}}\right)^p}.
\]
Hence by \eqref{sperbine} we get the conclusion.
\end{proof}

\begin{proof}[Proof of Theorem \ref{main2}] 
Let be $\beta<0$, and consider the positive solution $X$ of \eqref{1dim} corresponding to $\mu_{1}(\beta,s_0)$ given by Theorem \ref{mu_thm}, with $s_{0}=R_{F}(\Omega)$. 
Let us set
\[
z(t):=X(s_{0}-t),\quad t\in[0,s_{0}],\qquad v(x):=z(d_{F}(x))\quad\text{in}\ \Omega.
 \]
We claim that
\begin{equation}
\label{claimbetaneg}
\begin{cases}
 -\mathcal Q_{p}v \le \mu_{1}(\beta,s_0) v^{p-1}&\text{in }\Omega\\
 F(\nabla v)^{p-1} F_\xi(\nabla v)\cdot \nu_\Omega +\beta  F(\nu_\Omega) v^{p-1}\le 0&\text{on }\de \Omega
\end{cases}
\end{equation}
in weak sense. To prove the claim, we first observe that $z$ is a decreasing function, and recalling that $F(\nabla d_{F}(x))=1$, we have that
\begin{multline*}
\int_\Omega F^{p-1}(\nabla v) F_{\xi}(\nabla v)\cdot\nabla\varphi \,dx+\beta\int_{\partial\Omega} v^{p-1} F(\nu_\Omega) \varphi\, d\mathcal H^{N-1}-\mu_{1}(\beta,s_0) \int_\Omega v^{p-1}\varphi \,dx  \\
=-\int_\Omega[-z'(d_{F})]^{p-1}F_\xi(\nabla d_F)\cdot \nabla \varphi \, dx\\+\beta\int_{\partial\Omega} v^{p-1} F(\nu_\Omega) \varphi \, d\mathcal H^{N-1}- \mu_{1}(\beta,s_0)\int_\Omega v^{p-1}\varphi \,dx,
\end{multline*}
for any function $\varphi\in W^{1,p}(\Omega)$, $\varphi\ge0$ in $\Omega$. The function $[-z']^{p-1}F_\xi(\nabla d_F)\cdot \nabla \varphi$, in particular, is in $L^{1}(\Omega)$. Moreover, if $x=\Phi(y,t)$, $x\in \Omega$, then $t=F^{o}(x-y)=d_{F}(x)$. Hence, by Theorem \ref{change_normal} and \eqref{dentro_fuori}, it holds that
\begin{align*}
&\int_\Omega[-z'(d_{F}(x))]^{p-1}F_\xi(\nabla d_F(x))\cdot \nabla \varphi(x) \, dx \\ &=\int_{\partial\Omega} F(\nu_\Omega(y))\int_0^{\ell(y)}[-z'(t)]^{p-1}F_{\xi}(\nabla d_F(\Phi(y,t)))\cdot\nabla\varphi(\Phi(y,t))J(y,t)dt \,d\mathcal H^{N-1}(y) \\
&= \int_{\partial\Omega} F(\nu_\Omega(y))\int_0^{\ell(y)}[-z'(t)]^{p-1} \frac{d}{dt}\left[\varphi(\Phi(y,t))\right]J(y,t)\, dt \, d\mathcal H^{N-1}(y)
\end{align*}
Hence, integrating by parts and recalling that $\varphi(\Phi(y,0))=\varphi(y)$ and $J(y,0)=1$ one gets
\begin{align*}
&\int_0^{\ell(y)}[-z'(t)]^{p-1} \frac{d}{dt}\left[\varphi(\Phi(y,t))\right]J(y,t)\, dt \\[.2cm]
&=[-z'(\ell(y))]^{p-1}\varphi(\Phi(y,\ell(y)))J(y,\ell(y)) -
[-z'(0)]^{p-1}\varphi(y)  \\[.2cm]
&\;-\int_0^{\ell (y)}[-z'(t)]^{p-1}\varphi(\Phi(y,t))\frac{d}{dt}\left[J(y,t)\right]dt 
- \int_0^{\ell (y)}([-z'(t)]^{p-1})'\varphi(\Phi(y,t))J(y,t)dt \\[.2cm]
&\ge\beta z(0)^{p-1} \varphi(y) + \mu_{1}(\beta,s_{0}) \int_0^{\ell (y)}z(t)^{p-1}\varphi(\Phi(y,t))J(y,t)dt 
\end{align*}
where the inequality holds by the fact that $J>0$ and \eqref{laplaciano_cambio}. Then, by using again Theorem \ref{change_normal}, we have
\begin{multline*}
 \int_{\partial\Omega} F(\nu_\Omega(y))\int_0^{\ell(y)}[-z'(t)]^{p-1} \frac{d}{dt}\left[\varphi(\Phi(y,t))\right]J(y,t)\, dt \, d\mathcal H^{N-1}(y) \ge \\
\beta\int_{\de\Omega}F(\nu_{\Omega}) v^{p-1}\varphi\, d\mathcal H^{N-1} + \mu_{1}(\beta,s_{0})\int_{\Omega}v^{p-1}\varphi\, dx,
\end{multline*}
and so the claim \eqref{claimbetaneg} is proved. The proof of \eqref{sperbineN} is then completed by applying Lemma \ref{usefullemma}.
 \end{proof}
 \begin{proof}[Proof of Corollary \ref{corestN}]
It holds immediately from \eqref{sperbineN} and Remark \ref{argument}.
\end{proof}

In last part of this section we prove the optimality of the inequalities \eqref{sperbine} and \eqref{sperbineN} proved in Theorem \ref{main1}. \begin{prop} Let $\Omega_\ell=]-\frac a2,\frac a2[\times]-\frac \ell 2,\frac \ell 2[^{N-1}$, and suppose $R_{F}(\Omega_{\ell})=\frac a2 F^{o}(e_{1})$. 

If $\beta>0$, then
\[
\lim_{\ell\to +\infty}\frac{\lambda_1(\beta,\Omega_\ell)}{\mu_\ell}=1
\]
where $\mu_{\ell}$ is the first eigenvalue of \eqref{1dim} in $(0,s_{\ell})$, with $s_{\ell}=(M_{\ell}p')^{\frac{1}{p'}}$, $M_\ell=\max_{\Omega_\ell}w_{\Omega_\ell}$ and $w_{\Omega_\ell}$ the solution of problem \eqref{pbtord} in $\Omega_\ell$.

If $\beta<0$, then
\[
\lim_{\ell\to +\infty}\frac{\lambda_1(\beta,\Omega_\ell)}{\mu_\ell}=1
\]
where $\mu_{\ell}$ is the first eigenvalue of \eqref{1dim} in $(0,R_F(\Omega_{\ell}))$.
\end{prop}
\begin{proof}
Let us first consider the case when $\beta>0$. Being the inradius $R_{F}(\Omega)=\frac a 2 F^{o}(e_{1})$, it holds by Theorem \ref{paynemax} that 
\[
s_\ell=(p'M_\ell)^\frac{1}{p'}\leq \frac a2 F^{o}(e_{1})
\] 
and that $s_\ell\to \tilde a:=\frac a2 F^{o}(e_{1})$, $\mu_{\ell}\to \mu_{\tilde a}$ as $\ell\to+\infty$.

We set $\Omega_\ell=A_\ell \cup B_\ell$ where $A_\ell=]-\frac {s_\ell}{F^{o}(e_{1})},\frac {s_\ell}{F^{o}(e_{1})}[\times]-\frac \ell 2,\frac \ell 2[^{N-1}$ and $B_\ell=\Omega_{\ell}\setminus A_{\ell}$. 

We write $x=(x_1,x_2,...,x_{N-1},x_N)$ and consider the following function
\[
\varphi_\ell(x)=
\begin{cases}
X_{\ell}(|x_{1}|F^{o}(e_{1})) & \text{if }x\in A_\ell,\\
X_{\ell}(s_{\ell})& \text{if }x\in B_\ell,\\
\end{cases}
\]
where 
\[
X_{\ell}(s)=\cos_p\left(\left(\frac{\mu_\ell}{p-1}\right)^\frac{1}{p} s \right),\quad s\in [0,s_{\ell}].
\]
Therefore, by using \eqref{sperbine} and estimating $\lambda_1(\beta,\Omega_\ell)$ with the test function $\varphi_\ell$, we have
\[
\begin{split}
\mu_{\ell}\le\lambda_1(p,\Omega_\ell)&\leq\frac{\displaystyle\int_{\Omega_\ell} F(\nabla \varphi_\ell)^p dx +\beta \ds\int_{\de\Omega_\ell}|\varphi_\ell|^p d\cH^{N-1} }{\displaystyle\int_{\Omega_\ell}|\varphi_\ell|^p dx}.
\end{split}
\]
By straightforward computations, recalling that $F^{o}(e_{1})F(e_{1})=1$ (see \cite[Prop 4.4]{DPGGLB}), it holds that
\[
\frac{1}{\ell^{N-1}}\int_{\Omega_{\ell}}F(\nabla\varphi_{\ell})^{p}dx= \frac{2}{F^{o}(e_{1})} \int_{0}^{s_{\ell}}X_{\ell}'(s)^{p}ds,
\]
\begin{equation*}
\frac{1}{\ell^{N-1}}\int_{\de\Omega_{\ell}}\varphi_{\ell}^{p}F(\nu)dx=\frac{2}{F^{o}(e_{1})} X_{\ell}^{p}(s_{\ell})+O\left(\frac1\ell\right)
\end{equation*}
and
\[
\frac{1}{\ell^{N-1}}\int_{\Omega_{\ell}}\varphi^{p}dx = \frac{2}{F^{o}(e_{1})}\int_{0}^{s_{\ell}}X_{\ell}(s)^{p}ds + 4\left(\frac a2-\frac{s_{\ell}}{F^{o}(e_{1})}\right)X_{\ell}^{p}(s_{\ell}).
\]
Hence
\[
\lim_{\ell\to+\infty}\frac{\lambda_{1}(\beta,\Omega_{\ell})}{\mu_{\ell}}=1.
\]
The case $\beta<0$ analogously follows by considering the function
\[
\varphi_\ell(x)=
X_{\ell}(|x_{1}|F^{o}(e_{1})),\quad x\in \Omega_\ell,
\]
where
\[
X_{\ell}(s)=\cosh_p\left(\left(\frac{-\mu_\ell}{p-1}\right)^\frac{1}{p} s \right),\quad s\in [0,R_F(\Omega_{\ell})].
\]
\end{proof}

\begin{rem}
The inequality in Corollary \ref{corestN} has been obtained, in the Euclidean case, also in \cite[$p=2$, Theorem 2.3]{gs} or \cite[$1<p<+\infty$, Proposition 6.2]{KP} for any bounded Lipschitz domain, by applying a simple method based on the use of an exponential test function. It seems that the same method, in our setting, gives only a weaker form. Indeed, it is not difficult to show that plugging $\psi(x)=e^{\alpha x_i}$ in \eqref{eigint}, with $\alpha= \left[\frac{-\beta}{F^{o}(e_{i})F(e_{i})^{p}}\right]^{\frac{1}{p-1}}$, it holds that
\begin{equation}
\label{corestNweak}
\lambda_{1}(\beta,\Omega)\le (1-p)\left[\frac{-\beta}{F^{o}(e_{i})F(e_{i})}\right]^{\frac{p}{p-1}}, 
\end{equation}
for any $i=1,\ldots,N$, where $e_i$ is the $i-th$ vector of the standard basis in $\R^{N}$. Actually, the choice of $\alpha$ gives the best possible constant. Due to \eqref{prodscal}, it holds that $F^{o}(e_{i})F(e_{i})\ge 1$, so the constant in \eqref{corestNweak} is larger than the one in \eqref{corestNine}.
\end{rem}

\section*{Acknowledgements}
This work has been partially supported by the  MIUR-PRIN 2017 grant ``Qualitative and quantitative aspects of nonlinear PDE's'', by GNAMPA of INdAM, by  the FRA Project (Compagnia di San Paolo and Universit\`a degli studi di Napoli Federico II) \verb|000022--ALTRI_CDA_75_2021_FRA_PASSARELLI|.

{\small

\bibliographystyle
{plain}

}
\end{document}